\documentclass[12pt]{article}
\usepackage{a4wide}
\usepackage{amsthm}
\usepackage{amsfonts}
\usepackage{amssymb}
\usepackage{amsmath}
\usepackage{cite}
\usepackage{epsfig}
\usepackage{stmaryrd}
\newtheorem{theorem}{Theorem}
\newtheorem{lemma}[theorem]{Lemma}
\newcommand{\unlab}[2]{\left\llbracket #1\right\rrbracket_{#2}}
\newcommand\HH{{\cal H}}
\newcommand\MM{{\cal M}}
\newcommand\NN{{\mathbb N}}

\newcommand\RR{{\mathbb R}}
\newcommand\dd{\mbox{d}}

\newcommand{\OPleft}{\left(}
\newcommand{\OPright}{\right)}

\begin{document}
\title{Forcing Generalized Quasirandom Graphs Efficiently\thanks{An extended abstract announcing the results presented in this paper has been published in the Proceedings of Eurocomb'23.}}

\author{Andrzej Grzesik\thanks{Faculty of Mathematics and Computer Science, Jagiellonian University, {\L}ojasiewicza 6, 30-348 Krak\'{o}w, Poland. E-mail: {\tt Andrzej.Grzesik@uj.edu.pl}. Supported by the National Science Centre grant number 2021/42/E/ST1/00193.}
\and
Daniel Kr{\'a}l'\thanks{Faculty of Informatics, Masaryk University, Botanick\'a 68A, 602 00 Brno, Czech Republic. E-mail: {\tt kral@fi.muni.cz}. Supported by the MUNI Award in Science and Humanities (MUNI/I/1677/2018) of the Grant Agency of Masaryk University.}\and
Oleg Pikhurko\thanks{Mathematics Institute and DIMAP, University of Warwick, Coventry CV4 7AL, United Kingdom. E-mail: {\tt o.pikhurko@warwick.ac.uk}. Supported by  ERC Advanced Grant 101020255 and Leverhulme Research Project Grant RPG-2018-424.}}

\date{}

\maketitle

\begin{abstract}
We study generalized quasirandom graphs
whose vertex set consists of $q$ parts (of not necessarily the same sizes)
with edges within each part and between each pair of parts distributed quasirandomly;
such graphs correspond to the stochastic block model studied in statistics and network science.
Lov\'asz and S\'os showed that
the structure of such graphs is forced by homomorphism densities of graphs with at most $(10q)^q+q$ vertices;
subsequently, Lov\'asz refined the argument to show that graphs with $4(2q+3)^8$ vertices suffice.
Our results imply that the structure of generalized quasirandom graphs with $q\ge 2$ parts
is forced by homomorphism densities of graphs with at most $4q^2-q$ vertices, and,
if vertices in distinct parts have distinct degrees, then $2q+1$ vertices suffice.
The latter improves the bound of $8q-4$ due to Spencer.
\end{abstract}

\section{Introduction}
\label{sec:intro}

Quasirandom graphs play an important role in structural and extremal graph theory.
The notion of quasirandom graphs can be traced to the works of R\"odl~\cite{Rod86}, Thomason~\cite{Tho87,Tho87b} and
Chung, Graham and Wilson~\cite{ChuGW89} in the 1980s, and
is also deeply related to Szemer\'edi's Regularity Lemma~\cite{SimS91}.
Indeed,
the Regularity Lemma asserts that
each graph can be approximated by partitioning it into a bounded number of quasirandom bipartite graphs.
There is also a large body of literature concerning quasirandomness of various kinds of combinatorial structures
such as groups~\cite{Gow08},
hypergraphs~\cite{ChuG90,ChuG91s,Gow06,Gow07,HavT89,KohRS02,NagRS06,RodS04},
permutations~\cite{ChaKNPSV20,Coo04,KraP13,Kur22},
Latin squares~\cite{CooKLM22,EbeMM22,GarHHS20,GowL20},
subsets of integers~\cite{ChuG92},
tournaments~\cite{BucLSS19,ChuG91,CorPS19,CorR17,HanKKMPSV23,GrzIKK22}, etc.
Many of these notions have been treated in a unified way in the recent paper by Coregliano and Razborov~\cite{CorR20}.

The starting point of our work is the following classical result  on quasirandom graphs~\cite{ChuGW89}:
a sequence of graphs $(G_n)_{n\in\NN}$ is quasirandom with density $p$ if and only if
the homomorphism densities of the single edge $K_2$ and the 4-cycle $C_4$ in $(G_n)_{n\in\NN}$ converge to $p$ and $p^4$,
i.e., to their expected densities in the Erd\H os-R\'enyi random graph with density~$p$.
In particular, quasirandomness is forced by homomorphism densities of graphs with at most $4$ vertices.
In this paper, we consider a generalization of quasirandom graphs,
which corresponds to the stochastic block model in statistics.
In this model, the edge density of a (large) graph is not homogeneous as in the Erd\H os-R\'enyi random graph model, however,
the graph can be partitioned into $q$ parts such that
the edge density is homogeneous inside each part and between each pair of the parts.
Lov\'asz and S\'os~\cite{LovS08} established that
the structure of such graphs is forced by homomorphism densities of graphs with at most $(10q)^q+q$ vertices.
Lov\'asz~\cite[Theorem 5.33]{Lov12} refined this result by showing that
homomorphism densities of graphs with at most $4(2q+3)^8$ vertices suffice.
Our main result,
which we state below (we refer to Section~\ref{sec:notation} for not yet defined notation),
improves this bound:
the structure of generalized quasirandom graphs with $q\ge 2$ parts
is forced by homomorphism densities of graphs with at most $4q^2-q$ vertices.
\begin{theorem}
\label{thm:main-graphon}
The following holds for every $q\ge 2$ and every $q$-step graphon $W$:
if the density of each graph with at most $4q^2-q$ vertices in a graphon $W'$ is the same as in $W$,
then the graphons $W$ and $W'$ are weakly isomorphic.
\end{theorem}
We remark that 
our line of arguments to prove Theorem~\ref{thm:main-graphon}
substantially differs from that in~\cite{LovS08,Lov12},
with the exception of initial application of Lemma~\ref{lm:force_steps}.
In particular,
the key steps in our proof are more explicit and so of a more constructive nature,
which is of importance in relation to applications~\cite{BorCCG21,GaoLZ15,KloTV17as,KloV19}.

Spencer~\cite{Spe10} considered generalized quasirandom graphs with $q$ parts 
with an additional assumption that
vertices in distinct parts have distinct degrees, and
established that the structure of such graphs
is forced by homomorphism densities of graphs with at most $8q-4$ vertices.
Addressing a question posed in~\cite{Spe10},
we show (Theorem~\ref{thm:different}) that
graphs with at most $2q+1$ vertices suffice in this restricted setting for any $q\ge 2$.

We present our results and arguments using the language of the theory of graph limits,
which is introduced in Section~\ref{sec:notation}.
We remark that similarly to arguments presented in~\cite{LovS08,Lov12}, although not explicitly stated there,
our arguments also apply in a more general setting of kernels in addition to graphons (see Section~\ref{sec:notation}
for the definitions of the two notions).
We present various auxiliary results in Section~\ref{sec:steps} and use them to prove our main result 
in
Section~\ref{sec:main}.
The case with the additional assumption that vertices in distinct parts have distinct degrees
is analyzed in Section~\ref{sec:different}.

\section{Notation}
\label{sec:notation}

We now introduce the notions and tools from the theory of graph limits that we need in our arguments;
we refer the reader to the monograph by Lov\'asz~\cite{Lov12} for a more comprehensive introduction and further details.
We also rephrase results concerning quasirandom graphs and generalized quasirandom graphs with $q$ parts
presented in Section~\ref{sec:intro} in the language of the theory of graph limits.

We start with fixing some general shorthand notation used throughout the paper.
The set of the first $q$ positive integers is denoted by $[q]$ and
more generally the set of integers between $a$ and $b$ (inclusive) is denoted by $[a,b]$.
If $H$ and $G$ are two graphs,
the \emph{homomorphism density} of $H$ in $G$, denoted by $t(H,G)$,
is the probability that a random function $f:V(H)\to V(G)$,
with all $|V(G)|^{|V(H)|}$ choices being equally likely,
is a \emph{homomorphism} of $H$ to $G$, i.e., $f(u)f(v)$ is an edge of $H$ for every edge $uv$ of $G$.
A sequence $(G_n)_{n\in\NN}$ of graphs is \emph{convergent}
if the number of vertices of $G_n$ tends to infinity and the values of $t(H,G_n)$ converge for every graph $H$ as $n\to\infty$.
A sequence $(G_n)_{n\in\NN}$ of graphs is \emph{quasi\-random with density $p$}
if it is convergent and the limit of $t(H,G_n)$ is equal to $p^{|E(H)|}$ for every graph $H$,
where $E(H)$ denotes the edge set of $H$.
If the particular value of $p$ is irrelevant or understood,
we just say that a sequence of graphs is \emph{quasirandom} instead of quasirandom with density $p$.

The theory of graph limits provides analytic ways of representing sequences of convergent graphs.
A \emph{kernel} is a bounded measurable function $U:[0,1]^2\to\RR$ that
is \emph{symmetric}, i.e., $U(x,y)=U(y,x)$ for all $(x,y)\in [0,1]^2$.
A \emph{graphon} is a kernel whose values are restricted to~$[0,1]$.
The \emph{homomorphism density} of a graph $H$ in a kernel $U$ is defined as follows:
\[t(H,U)=\int_{[0,1]^{V(H)}}\prod_{uv\in E(H)}U(x_u,x_v)\dd x_{V(H)},\]
 where $\dd x_A$ for a set $A=\{a_1,\dots,a_k\}$ is a shorthand for $\dd x_{a_1}\dots\dd x_{a_k}$;
we often just briefly say the \emph{density} of a graph $H$ in a kernel $U$ rather than
the homomorphism density of $H$ in $U$.
A graphon $W$ is a \emph{limit} of a convergent sequence $(G_n)_{n\in\NN}$ of graphs
if $t(H,W)$ is the limit of $t(H,G_n)$ for every graph $H$.
Every convergent sequence of graphs has a limit graphon and
every graphon is a limit of a convergent sequence of graphs as shown by Lov\'asz and Szegedy~\cite{LovS06};
also see~\cite{DiaJ08} for relation to exchangeable arrays.
Two kernels (or graphons) $U_1$ and $U_2$ are \emph{weakly isomorphic} if $t(H,U_1)=t(H,U_2)$ for every graph $H$.
Note that any two limits of the same convergent sequence of graphs are weakly isomorphic, and
we refer particularly to~\cite{BorCL10} for results on the structure of weakly isomorphic graphons and more generally kernels.

We phrase the results concerning quasirandom graphs using the language of the theory of graph limits.
Observe that a sequence of graphs is quasirandom if and only if
it converges to the graphon equal to $p$ everywhere.
The following holds for every graphon $W$ and every real $p\in [0,1]$:
a graphon $W$ is weakly isomorphic to the constant graphon equal to $p$ if and only if $t(K_2,W)=p$ and $t(C_4,W)=p^4$.
This leads us to the following definition:
a graphon $U$ is \emph{forced} by graphs contained in a set $\HH$
if every graphon $U'$ such that $t(H,U')=t(H,U)$ for every graph $H\in\HH$ is weakly isomorphic to $U$.
In particular, any constant graphon is forced by the graphs $K_2$ and $C_4$.
We refer particularly to~\cite{CooKM18,GrzKL20,KraLNS20,LovS11} for results on the structure of graphons
forced by finite sets of graphs.
Similarly, we say that a kernel $U$ is \emph{forced} by graphs from a set $\HH$
if every kernel $U'$ such that $t(H,U')=t(H,U)$ for every graph $H\in\HH$ is weakly isomorphic to $U$.
We emphasize that our results actually concern forcing kernels (rather than graphons),
which makes them formally stronger.

A \emph{$q$-step kernel $U$}
is a kernel such that $[0,1]$ can be partitioned into $q$ non-null measurable sets $A_1,\ldots,A_q$ such that
$U$ is constant on $A_i\times A_j$ for all $i,j\in [q]$ but there is no such partition with $q-1$ parts.
A \emph{$q$-step graphon} is a $q$-step kernel that is also a graphon.
If the number of parts is not important, we use a \emph{step kernel} or a \emph{step graphon} for brevity.
Observe that step graphons correspond to stochastic block models and
so to generalized quasirandom graphs discussed in Section~\ref{sec:intro}.
As mentioned in Section~\ref{sec:intro},
Lov\'asz and S\'os~\cite[Theorem~2.3]{LovS08} showed that
every $q$-step graphon $W$ is forced by graphs with at most $(10q)^q+q$ vertices, and
Lov\'asz~\cite[Theorem 5.33]{Lov12} further improved the bound on the number of vertices to $4(2q+3)^8$;
we remark that the proof of either of the results can be adapted to the setting of step kernels.
Our main result (Theorem~\ref{thm:main-graphon}) states that
every $q$-step graphon is forced by graphs with at most $\max\{4q^2-q,4\}$ vertices;
our arguments also apply in the setting of step kernels as stated in Theorem~\ref{thm:main-kernel}.

In the rest of this section, we introduce some technical notation needed to present our arguments.
A \emph{$k$-rooted graph} is a graph with $k$ distinguished pairwise distinct vertices, and
more generally an \emph{$(s_1,\ldots,s_q)$-rooted graph} is an $(s_1+\cdots+s_q)$-rooted graph
whose roots are split into $q$ groups, each of size $s_i$, $i\in [q]$.
If $H$ is a $k$-rooted graph with vertices $v_1,\ldots,v_n$ such that its roots are $v_1,\ldots,v_k$ then
the \emph{density} of $H$ in a kernel $U$ when $x_1,\ldots,x_k\in [0,1]$ are chosen as the roots is defined as:
\[t_{x_1,\ldots,x_k}(H,U)=\int_{[0,1]^{n-k}}\prod_{v_iv_j\in E(H)}U(v_i,v_j)\dd x_{[k+1,n]}.\]
By the Fubini--Tonelli Theorem,
the integral exists for almost all choices of $x_1,\ldots,x_k$ and we will often ignore exceptional null-sets in this paper.
Note that for $k=0$ this definition coincides  with the definition of the density of an unrooted graph in a kernel.
If the particular choice of the roots is understood, we write $t_{\star}(H,U)$ instead of $t_{x_1,\ldots,x_k}(H,U)$.
We sometimes think of and refer to the elements of $[0,1]$ as  \emph{vertices} of a kernel,
which justifies the definition of the density of a rooted graph in a kernel and
leads to the following definition:
the \emph{degree} of a vertex $x\in [0,1]$ in a kernel $U$ is the density $t_x(K_2^{\bullet},U)=\int_0^1 U(x,y)\dd y$,
where $K_2^{\bullet}$ is the $1$-rooted graph obtained from $K_2$ by choosing one of its vertices as the root.

A \emph{quantum graph} is a formal finite linear combination $Q=\sum_{i=1}^m c_i H_i$ of graphs; a graph $H_i$ with $c_i\not=0$ is called a \emph{constituent} of~$Q$. 
More generally
a \emph{quantum $k$-rooted graph} is a formal finite linear combination of $k$-rooted graphs such that
their roots induce the same ($k$-vertex) subgraph in each of the constituents.
The \emph{density} of a (rooted) quantum graph $Q$ in a kernel $U$
is the corresponding linear combination of the densities of the constituents forming~$Q$. 

For a $k$-rooted graph $H$, let $\unlab{H}{}$ be the underlying unrooted graph.
Note that
it holds for every kernel $U$ that
\[t(\unlab{H}{},U)=\int_{[0,1]^k}t_{x_1,\ldots,x_k}(H,U)\dd x_{[k]}.\]
If $H$ and $H'$ are $k$-rooted graphs such that
every pair of corresponding roots is joined by an edge in at most one of the graphs $H$ and $H'$,
we define the \emph{product} $H\times H'$ as follows:
let $H''$ be the $k$-rooted graph isomorphic to $H'$ that has the same roots as $H$ and is vertex disjoint otherwise, and
let $H\times H'$ be the graph with the vertex set $V(H)\cup V(H'')$, the edge set $E(H)\cup E(H')$ and the same set of roots.
Note that $H\times H'$ does not have parallel edges
as each pair of corresponding roots is joined by an edge in at most one of the graphs $H$ and $H'$.
Also observe that 
$
|V(H\times H')|=|V(H)|+|V(H')|-k
$
and it holds for every choice of roots and every kernel $U$ that
\[t_{\star}(H\times H',U)=t_{\star}(H,U)\cdot t_{\star}(H',U).\]
If $H=H'$, we may write $H^2$ instead of $H\times H$.
The definition of the operator $\unlab{\cdot}{}$ and that of the product extend to rooted quantum graphs by linearity.
Observe that, for every $k$-rooted quantum graph $Q$ and every kernel $U$,
it holds that $t(\unlab{Q^2}{},U)\ge 0$ and
the equality holds if and only if $t_{\star}(Q,U)=0$ for almost every choice of roots. 

\section{Forcing step structure}
\label{sec:steps}

We start with recalling a construction from~\cite[Proposition 14.44]{Lov12},
which forces the structure of a step kernel with at most $q$ parts.
For $k\in\NN$ and $1\le i<j\le k$,
let $Q_k^{ij}$ be the following $(2k)$-rooted quantum graph with roots $v_1,\ldots,v_k$ and $v'_1,\ldots,v'_k$.
The quantum graph $Q_k^{ij}$ has four constituents, each with a single non-root vertex:
the graph with the non-root vertex adjacent to $v_i$ and $v'_i$ and
the graph with the non-root vertex adjacent to $v_j$ and $v'_j$, both with coefficient $+1$,
as well as the graph with the non-root vertex adjacent to $v_i$ and $v'_j$ and
the graph with the non-root vertex adjacent to $v_j$ and $v'_i$, both with coefficient $-1$.
See Figure~\ref{fig:Qkij} for an example.
Let $Q_k$ be the following quantum graph with each constituent having $2k+2\binom{k}{2}=k(k+1)$ vertices:
\[Q_k=\unlab{\prod_{1\le i<j\le k} \left(Q_k^{ij}\right)^2}{}.\]
The graph $Q_k$ is the graph obtained in the proof of~\cite[Proposition 14.44]{Lov12}
through an application~\cite[Lemma 14.37]{Lov12}. This gives the following lemma, whose proof we sketch for completeness. 

\begin{figure}
\begin{center}
\epsfbox{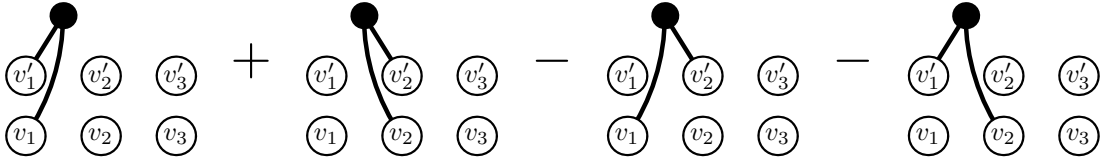}
\end{center}
\caption{The $6$-rooted quantum graph $Q_3^{12}$.}
\label{fig:Qkij}
\end{figure}

\begin{lemma}
\label{lm:force_steps}
For every $q\in\NN$ and every kernel $U$, the following holds:
$t(Q_{q+1},U)=0$ if and only if $U$ is weakly isomorphic to a step kernel with at most $q$ parts.
\end{lemma}

\begin{proof}
Observe that the value of $t(Q_{q+1},U)$ for a kernel $U$ is equal to 
\begin{equation}
\int_{[0,1]^{2(q+1)}}\prod_{1\le i<j\le q+1}\!\!\left(\int_{[0,1]}\big(U(x_i,y)-U(x_j,y)\big)\big(U(x'_i,y)-U(x'_j,y)\big)\dd y\right)^2\!\!\dd x_{[q+1]} \dd x'_{[q+1]}.
\label{eq:Qq}
\end{equation}
If $U$ is a step kernel with at most $q$ parts,
then for any choice of $x_1,\ldots,x_{q+1}$,
there exist $1\le i<j\le q+1$ such that $x_i$ and $x_j$ are from the same part of $U$ and
so $U(x_i,y)=U(x_j,y)$ for all $y\in [0,1]$.
Consequently, the product in \eqref{eq:Qq} is zero for any choice of roots $x_1,\ldots,x_{q+1}$,
which implies that $t(Q_{q+1},U)=0$.

We now prove the other implication, i.e., that
if $t(Q_{q+1},U)=0$, then $U$ is weakly isomorphic to a step kernel with at most $q$ parts.
Let $U$ be a kernel such that $t(Q_{q+1},U)=0$.
By~\eqref{eq:Qq}, the following holds for almost all $x_{[q+1]}\in [0,1]^{q+1}$ and $x'_{[q+1]}\in [0,1]^{q+1}$:
\[
\prod_{1\le i<j\le q+1}\int_{[0,1]}\left(U(x_i,y)-U(x_j,y)\right)\left(U(x'_i,y)-U(x'_j,y)\right)\dd y=0.
\]
Using~\cite[Proposition 13.23]{Lov12}, we get that
the following holds for almost all $x_{[q+1]}\in [0,1]^{q+1}$:
\begin{equation}
\prod_{1\le i<j\le q+1}\int_{[0,1]}\left(U(x_i,y)-U(x_j,y)\right)^2\dd y=0.
\label{eq:Qq2}
\end{equation}
Let us consider an equivalence relation on $[0,1]$ defined as $x\equiv x'$ if $U(x,y)=U(x',y)$ for almost all $y\in [0,1]$.
Observe that \eqref{eq:Qq2} holds for $x_{[q+1]}\in [0,1]^{q+1}$ if and only if
there exist $1\le i<j\le q+1$ such that $x_i\equiv x_j$.
Hence, \eqref{eq:Qq2} holds for almost all $x_{[q+1]}\in [0,1]^{q+1}$ if and only if
the measure of the $q$ largest equivalence classes of $\equiv$ is one,
which is equivalent to $U$ being weakly isomorphic to a step kernel with at most $q$ parts.
\end{proof}

We next present two rather similar auxiliary lemmas;
since their statements and constructions somewhat differ depending on the parity of $q$,
we state them separately for readability.

\begin{lemma}
\label{lm:color_even}
For every even integer $q\ge 2$ and all integers $s_1,\ldots,s_q\in [q+2,2q+2]$,
there exists a graph $G$ with vertex set formed by $q$ disjoint sets $V_1,\ldots,V_q$ that
satisfies the following:
\begin{itemize}
\item the size of $V_i$ is $s_i$ for each $i\in [q]$,
\item the edge set of $G$ can be partitioned into four sets $M_1,\ldots,M_4$ such
      that, for every $1\le i<j\le q$,
      each of the sets $M_1$ and $M_2$ restricted to vertices of $V_i\cup V_j$, 
      is a matching of size $q+2$, and
      each of the sets $M_3$ and $M_4$ is a matching of size $q$, and
\item the chromatic number of $G$ is $q$ and
      the color classes of every $q$-coloring of $G$ are precisely the sets $V_1,\ldots,V_q$;
      in particular, each of the sets $V_i$, $i\in [q]$, is independent.
\end{itemize}
\end{lemma}

\begin{proof}
Fix an even integer $q\ge 2$ and integers $s_1,\ldots,s_q\in [q+2,2q+2]$.
Let $V_i=\{i\}\times [s_i]$;
note that the first coordinate of a vertex determines which of the sets contains the vertex.
We now describe the graph $G$ by listing the edges between $V_i$ and $V_j$, $1\le i<j\le q$,
contained in the matchings $M_1,\ldots,M_4$, where we abbreviate $\{(a,b),(c,d)\}$ to $(a,b)(c,d)$.
\begin{itemize}
\item The matching $M_1$ consists of the edge $(i,1)(j,1)$, the edge $(i,q+2)(j,q+2)$, and
      the edges $(i,k)(j,k+1)$ and $(i,k+1)(j,k)$ for even values $k$ between $2$ and $q$.
\item The matching $M_2$ consists of the edges $(i,k)(j,k+1)$ and $(i,k+1)(j,k)$
      for odd values $k$ between $1$ and $q+1$.
\item The matching $M_3$ consists of the edges $(i,k)(j,s_j-q+k)$ for all $k\in [q]$.
\item The matching $M_4$ consists of the edges $(i,s_i-q+k)(j,k)$ for all $k\in [q]$.
\end{itemize}
Observe that the following edges are always present between $V_i$ and $V_j$, $1\le i<j\le q$:
\begin{itemize}
\item the edges $(i,1)(j,1)$,
\item the edges $(i,k)(j,k+1)$ and $(i,k+1)(j,k)$ for $k\in [q+1]$, and
\item the edges $(i,k)(j,s_j-q+k)$ and $(i,s_i-q+k)(j,k)$ for $k\in [q]$.
\end{itemize}
Since the sets $V_1,\ldots,V_q$ are independent, the chromatic number of $G$ is at most $q$.
On the other hand, the vertices $(i,1)$, $i\in [q]$ form a complete graph of order $q$,
which implies that the chromatic number of $G$ is at least $q$ and so it is equal to $q$.

Consider an arbitrary $q$-coloring of $G$ and
let $W_i$, $i\in [q]$, be the color class containing the vertex $(i,1)$. (Note that the vertices
$(i,1)$, $i\in [q]$, are colored with distinct colors as they form a complete graph.)
We prove the following statement by induction on $k$:
for every $i\in [q]$, if $k\le s_i$, then the vertex $(i,k)$ belongs to $W_i$.
If $k=1$, the statement follows from the definition of the sets $W_i$.
If $k\in [2,q+2]$,
for every $i\in [q]$,
the existence of the edges $(j,k-1)(i,k)$, $j\in [q]\setminus\{i\}$, and
the induction assumption, which states that $(j,k-1)$ belongs to $W_j$ for $j\not=i$, imply that
the vertex $(i,k)$ belongs to $W_i$.
Finally, if $k\in [q+3,s_i]$, $i\in [q]$,
the existence of the edges $(j,q+k-s_i)(i,k)$, $j\in [q]\setminus\{i\}$, implies that
the vertex $(i,k)$ belongs to $W_i$ (note that $q+k-s_i\le q$ and so $(j,q+k-s_i)\in W_j$ for $j\not=i$).
Hence, the $q$-coloring of $G$ is unique up to a permutation of color classes.
\end{proof}

We next present the version of Lemma~\ref{lm:color_odd} for odd values of $q\ge 3$.

\begin{lemma}
\label{lm:color_odd}
For every odd integer $q\ge 3$ and all integers $s_1,\ldots,s_q\in [q+2,2q+2]$,
there exists a graph $G$ with vertex set formed by $q$ disjoint sets $V_1,\ldots,V_q$ that
satisfies the following:
\begin{itemize}
\item the size of $V_i$ is $s_i$ for each $i\in [q]$,
\item the edge set of $G$ can be partitioned into four sets $M_1,\ldots,M_4$ such
      that each of the sets $M_1,\ldots,M_4$ restricted to vertices of $V_i\cup V_j$, $1\le i<j\le q$,
      is a matching of size $q+1$, and
\item the chromatic number of $G$ is $q$ and
      the color classes of every $q$-coloring of $G$ are precisely the sets $V_1,\ldots,V_q$;
      in particular, each of the sets $V_i$, $i\in [q]$, is independent.
\end{itemize}
\end{lemma}

\begin{proof}
Fix an odd integer $q\ge 3$ and integers $s_1,\ldots,s_q\in [q+2,2q+2]$, and set $V_i=\{i\}\times [s_i]$.
We describe the graph $G$ by listing the edges between $V_i$ and $V_j$, $1\le i<j\le q$,
contained in the matchings $M_1,\ldots,M_4$ (the definition of $M_3$ and $M_4$ needs to be altered
if $s_j=q+2$ or $s_i=q+2$, respectively, so that all four matchings $M_1,\ldots,M_4$ are disjoint).
\begin{itemize}
\item The matching $M_1$ consists of the edge $(i,1)(j,1)$, the edge $(i,q+1)(j,q+1)$, and
      the edges $(i,k)(j,k+1)$ and $(i,k+1)(j,k)$ for even values $k$ between $2$ and $q-1$.
\item The matching $M_2$ consists of the edges $(i,k)(j,k+1)$ and $(i,k+1)(j,k)$
      for odd values $k$ between $1$ and $q$.
\item The matching $M_3$ consists of the edges $(i,k)(j,s_j-q-1+k)$ for all $k\in [q+1]$ unless $s_j=q+2$;
      if $s_j=q+2$, then the matching $M_3$ consists of the edges $(i,q+1)(j,q+2)$, $(i,q+2)(j,2)$ and $(i,k)(j,k+2)$ for $k\in [q-1]$.
\item The matching $M_4$ consists of the edges $(i,s_i-q-1+k)(j,k)$ for all $k\in [q+1]$ unless $s_i=q+2$;
      if $s_i=q+2$, then the matching $M_4$ consists of the edges $(i,q+2)(j,q+1)$, $(i,2)(j,q+2)$ and $(i,k+2)(j,k)$ for $k\in [q-1]$.
\end{itemize}
Observe that the following edges are always present between $V_i$ and $V_j$, $1\le i<j\le q$:
\begin{itemize}
\item the edges $(i,1)(j,1)$,
\item the edges $(i,k)(j,k+1)$ and $(i,k+1)(j,k)$ for $k\in [q]$,
\item the edges $(i,k)(j,s_j-q-1+k)$ for $k=2q+3-s_j,\ldots,q+1$, and
\item the edges $(i,s_i-q-1+k)(j,k)$ for $k=2q+3-s_i,\ldots,q+1$.
\end{itemize}
The rest of the argument now follows the lines of the corresponding part of the proof of Lemma~\ref{lm:color_even}.
\end{proof}

We are now ready to prove the main lemma of this section.

\begin{lemma}
\label{lm:steps}
For all integers $q\ge 2$ and $s_1,\ldots,s_q\in [q+2,2q+2]$,
there exists an $(s_1,\ldots,s_q)$-rooted quantum graph $P_{s_1,\ldots,s_q}$ such that
\begin{itemize}
\item each constituent of $P_{s_1,\ldots,s_q}$ has $2q(q-1)$ non-root vertices, 
\item the $s_1+\cdots+s_q$ roots of $P_{s_1,\ldots,s_q}$ form an independent set,
\item for every $q$-step kernel $U$,
      there exists $d_0=d_0(U)>0$ that does not depend on $s_1,\ldots,s_q$ such that 
      $t_{\star}(P_{s_1,\ldots,s_q},U)$ is either $0$ or $d_0$ for all choices of roots, and
      it is non-zero if and only if
      all roots from each of the $q$ groups of roots of $P_{s_1,\ldots,s_q}$ are chosen from the same part of $U$
      but the roots from different groups are chosen from different parts.
\end{itemize}
\end{lemma}

\begin{proof}
For $q$ and $s_1,\ldots,s_q\in [q+2,2q+2]$,
let $G$ be the graph from Lemma~\ref{lm:color_even} or Lemma~\ref{lm:color_odd} (depending on the parity of $q$).
Let $V_1,\ldots,V_q$ be the sets forming the vertex set of $G$, and
let $M_1,\ldots,M_4$ be the sets forming the edge set of $G$ as given by the lemma.
We identify the vertices of $V_i$ with the $s_i$ roots in the $i$-th group.
Let $M^{ij}_k$, for $1\le i<j\le q$ and $k\in [4]$, consist of the edges of $M_k$ between $V_i$ and $V_j$, and
let $\MM^{ij}_k$ be the set of all $2^{\left|M^{ij}_k\right|}$ subsets of $V_i\cup V_j$ such that
each set in $\MM^{ij}_k$ contains exactly one vertex from each edge of $M^{ij}_k$.
Next, if $W\subseteq V_1\cup\cdots\cup V_q$, we write $P[W]$ for the $(s_1,\ldots,s_q)$-rooted graph
with a single non-root vertex such that the non-root vertex is adjacent to the roots in $W$.
Finally, we define the $(s_1,\ldots,s_q)$-rooted quantum graph $P_{s_1,\ldots,s_q}$ as follows:
\[P_{s_1,\ldots,s_q}=\prod_{1\le i<j\le q}\prod_{k\in [4]}\sum_{W\in\MM^{ij}_k} (-1)^{\left|W\cap V_i\right|}P[W].\]
Observe that each constituent of the quantum graph $P_{s_1,\ldots,s_q}$
has exactly $4\cdot\binom{q}{2}=2q(q-1)$ non-root vertices, and
the $s_1+\cdots+s_q$ roots form an independent set.
We remark that the $(s_1,\ldots,s_q)$-rooted quantum graph
\begin{equation}
\sum_{W\in\MM^{ij}_k} (-1)^{\left|W\cap V_i\right|}P[W]
\label{eq:Wprod}
\end{equation}
from the definition of $P_{s_1,\ldots,s_q}$ can also be obtained in the following alternative way,
which gives additional insight into the definition of $P_{s_1,\ldots,s_q}$.
Let $P'[v]$ be the $(s_1,\ldots,s_q,1)$-rooted graph such that 
$P'[v]$ has no non-root vertices, $v$ is a root contained in one of the first $q$ groups of roots, and
the only edge of $P'[v]$ is an edge joining the vertex $v$ and the single root contained in the last group.
For $1\le i<j\le q$ and $k\in [4]$,
the $(s_1,\ldots,s_q)$-rooted quantum graph \eqref{eq:Wprod} can be obtained
from the $(s_1,\ldots,s_q,1)$-rooted graph
\[\prod_{vu\in M^{ij}_k}\left(P'[v]-P'[u]\right)\]
by changing the single root contained in the last group to a non-root vertex.

For the rest of the proof, fix a $q$-step kernel $U$ and
let $z_i$, $i\in [q]$, be any vertex of $U$ contained in the $i$-th part of $U$.
Consider a choice $x_v$, $v\in V(G)$, of roots.
Suppose that $G$ has an edge $uv$ such that
$u\in V_i$, $v\in V_j$, $1\le i<j\le q$, $uv\in M_k$, $k\in [4]$, and
the vertices $x_u$ and $x_{v}$ belong to the same part of the kernel $U$.
Observe that
\begin{align*}
&\sum_{W\in\MM^{ij}_k} (-1)^{\left|W\cap V_i\right|}\prod_{w\in W}U(x_w,y)\\
= & \sum_{\substack{W\in\MM^{ij}_k\\u\in W}} (-1)^{\left|W\cap V_i\right|}\prod_{w\in W}U(x_w,y)+
    \sum_{\substack{W\in\MM^{ij}_k\\v\in W}} (-1)^{\left|W\cap V_i\right|}\prod_{w\in W}U(x_w,y)\\
= & \sum_{\substack{W\in\MM^{ij}_k\\u\in W}} (-1)^{\left|W\cap V_i\right|}\prod_{w\in W}U(x_w,y)+
    \sum_{\substack{W\in\MM^{ij}_k\\u\in W}} (-1)^{\left|W\cap V_i\right|-1}U(x_{v},y)\prod_{w\in W\setminus\{u\}}U(x_w,y)\\
= & \sum_{\substack{W\in\MM^{ij}_k\\u\in W}} (-1)^{\left|W\cap V_i\right|}\left(U(x_u,y)-U(x_{v},y)\right)\prod_{w\in W\setminus\{u\}}U(x_w,y)\\
= &\ 0.
\end{align*}
It follows that $t_{x_{V(G)}}(P_{s_1,\ldots,s_q},U)=0$ if
the coloring of the vertices of $G$ such that $v$ is colored with the part containing $x_v$
is not a proper coloring of $G$.
Either Lemma~\ref{lm:color_even} or Lemma~\ref{lm:color_odd} (depending on the parity of $q$) implies that
$t_{x_{V(G)}}(P_{s_1,\ldots,s_q},U)\not=0$ only if
all roots from each of the $q$ groups of roots are chosen from the same part of $U$ and
the roots from different groups are chosen from different parts.
If this is indeed the case and $q$ is odd,
the properties of the graph $G$ given in Lemma~\ref{lm:color_odd} imply that
\begin{equation}
t_{x_{V(G)}}(P_{s_1,\ldots,s_q},U)=\prod_{1\le i<j\le q}\left(\int_{[0,1]}\left(U(z_i,y)-U(z_j,y)\right)^{q+1}\dd y\right)^4.
\label{eq:d0_odd}
\end{equation}
This is positive
since for every distinct $i,j\in [q]$ there is a positive measure of $y$ with $U(z_i,y)\not= U(z_j,z)$ (as otherwise the $i$-th and $j$-th parts can be merged together contrary to the definition of a $q$-step kernel) and
$q+1$ is even.
Hence, the existence of $d_0$ follows and it is equal to the right hand side of \eqref{eq:d0_odd},
which does not depend on the values of $s_1,\ldots,s_q$.
Similarly, if $q$ is even,
the existence of $d_0$ follows from Lemma~\ref{lm:color_even} and the definition of $P_{s_1,\ldots,s_q}$, and
its value is 
\begin{equation}
d_0=\prod_{1\le i<j\le q}\left(\int_{[0,1]}\left(U(z_i,y)-U(z_j,y)\right)^{q+2}\dd y\right)^2
                         \left(\int_{[0,1]}\left(U(z_i,y)-U(z_j,y)\right)^{q}\dd y\right)^2.
\label{eq:d0_even}
\end{equation}
The proof of the lemma is now completed.
\end{proof}

We emphasize that the value of $d_0$ from the statement of Lemma~\ref{lm:steps} depends on the kernel $U$ only,
i.e., it does not depend on $s_1,\ldots,s_q$;
namely, $d_0$ is given by the right hand side of \eqref{eq:d0_odd} or \eqref{eq:d0_even} depending on the parity of $q$, the number of parts of the step kernel~$U$.

\section{Main result}
\label{sec:main}

We start with a construction of a quantum graph that restricts the \emph{density} of each part $A$ of a step kernel $U$, that is, the value of $U$ on $A\times A$.

\begin{lemma}
\label{lm:force_dens_inner}
For all integers $q\ge 2$, $k\in [q]$ and reals $d_1,\ldots,d_k$, 
there exists a quantum graph $R_{d_1,\ldots,d_k}$ such that
each constituent of $R_{d_1,\ldots,d_k}$ has $3q^2$ vertices and
the following holds for every $q$-step kernel $U$:
$t(R_{d_1,\ldots,d_k},U)=0$ if and only if
the density of each part of $U$ is one of the reals $d_1,\ldots,d_k$.
\end{lemma}

\begin{proof}
Fix $q\ge 2$ and reals $d_1,\ldots,d_k$.
Let $P_{q+2,\ldots,q+2}$ be the graph from Lemma~\ref{lm:steps}.
Note that $P_{q+2,\ldots,q+2}$ has $q(q+2)+2q(q-1)=3q^2$ vertices.
For $m\in [0,2k]$,
we set $P^{(m)}_{q+2,\ldots,q+2}$ to be a graph obtained from $P_{q+2,\ldots,q+2}$
by adding arbitrary $m$ edges among the roots in the first group (without creating parallel edges);
note that this is possible since $2k\le 2q\le\binom{q+2}{2}$.
Further, let $p(x)$ be the polynomial defined as
\[p(x)=\prod_{i=1}^k(x-d_i)^2,\]
and set $R_{d_1,\ldots,d_k}$ to be the quantum graph
obtained from the expansion of $p(x)$ into monomials by replacing each monomial $x^m$, including $x^0$, with $\unlab{P^{(m)}_{q+2,\ldots,q+2}}{}$.

Consider any $q$-step kernel $U$ and let $d_0=d_0(U)>0$ be the constant from Lemma~\ref{lm:steps}.
Observe that
\[t\left(\unlab{P^{(m)}_{q+2,\ldots,q+2}}{},U\right)=d_0(q-1)!\OPleft\prod_{i=1}^q a_i^{q+2}\OPright\OPleft\sum_{i=1}^q p_i^m\OPright,\]
where $a_i$ is the measure and $p_i$ is the density of the $i$-th part of $U$, $i\in [q]$;
note that the term $(q-1)!$ counts possible choices of parts of $U$ for the second, third, etc.\ group of roots
while the choices of the part for the first group of roots are accounted for by the last sum in the expression.
It follows that
\[t\left(R_{d_1,\ldots,d_k},U\right)=d_0(q-1)!\OPleft\prod_{i=1}^q a_i^{q+2}\OPright\OPleft\sum_{i=1}^q p(p_i)\OPright,\]
which, using $p(x)\ge 0$ for all $x\in\RR$, is equal to zero if and only if $p(p_i)=0$ for every $i\in [q]$.
The latter holds if and only if each $p_i$ is one of the reals $d_1,\ldots,d_k$ (note that $p(x)>0$ unless $x\in\{d_1,\ldots,d_k\}$), and
so the quantum graph $R_{d_1,\ldots,d_k}$ has the properties given in the statement of the lemma.
\end{proof}

The next lemma provides a quantum graph restricting densities between pairs of parts of a step kernel;
its proof is similar to that of Lemma~\ref{lm:force_dens_inner}, however,
we include it for completeness.

\begin{lemma}
\label{lm:force_dens_between}
For all integers $q\ge 2$, $k\in [q(q-1)/2]$ and reals $d_1,\ldots,d_k$, 
there exists a quantum graph $S_{d_1,\ldots,d_k}$ with $3q^2$ vertices such that
the following holds for every $q$-step kernel $U$:
$t(S_{d_1,\ldots,d_k},U)=0$ if and only if
the density between each pair of distinct parts of $U$ is one of the reals $d_1,\ldots,d_k$.
\end{lemma}

\begin{proof}
Fix $q\ge 2$ and reals $d_1,\ldots,d_k$.
Let $P_{q+2,\ldots,q+2}$ be the graph from Lemma~\ref{lm:steps}.
Recall that $P_{q+2,\ldots,q+2}$ has $q(q+2)+2q(q-1)=3q^2$ vertices.
For $m\in [0,2k]$,
we set $P^{(m)}_{q+2,\ldots,q+2}$ to be a graph obtained from $P_{q+2,\ldots,q+2}$
by adding arbitrary $m$ edges joining a root in the first group and a root in the second group without creating parallel edges;
note that this is possible since $2k\le q(q-1)\le (q+2)^2$.
Further, let $p(x)$ be the polynomial defined as
\[p(x)=\prod_{i=1}^k(x-d_i)^2,\]
and set $S_{d_1,\ldots,d_k}$ to be the quantum graph
obtained from the expansion of $p(x)$ by replacing $x^m$ with $\unlab{P^{(m)}_{q+2,\ldots,q+2}}{}$.

Consider a $q$-step kernel $U$ and let $d_0=d_0(U)>0$ be the constant from Lemma~\ref{lm:steps}.
Observe that
\[t\left(\unlab{P^{(m)}_{q+2,\ldots,q+2}}{},U\right)=2d_0(q-2)!\OPleft\prod_{i=1}^q a_i^{q+2}\OPright\OPleft\sum_{1\le i<j\le q} p_{ij}^m\OPright,\]
where $a_i$ is the measure of the $i$-th part of $U$, $i\in [q]$, and
$p_{ij}$ is the density between the $i$-th and $j$-th part of $U$, $1\le i<j\le q$.
It follows that
\[t\left(S_{d_1,\ldots,d_k},U\right)=2d_0(q-2)!\OPleft\prod_{i=1}^q a_i^{q+2}\OPright\OPleft\sum_{1\le i<j\le q}p\left(p_{ij}\right)\OPright,\]
which (by $p\ge 0$) is equal to zero if and only if $p(p_{ij})=0$ for all $1\le i<j\le q$.
The latter holds if and only if each $p_{ij}$, $1\le i<j\le q$, is one of the reals $d_1,\ldots,d_k$, and
so the quantum graph $S_{d_1,\ldots,d_k}$ has the properties given in the statement of the lemma.
\end{proof}

We next present a construction of a rooted quantum graph that ``tests''
whether there is a permutation of parts of a step kernel matching densities in a given matrix~$D$.
As the value of $d_0$ in Lemma~\ref{lm:steps},
the value of $c_0$ in Lemma~\ref{lm:force_dens_part} does not depend on $s_1,\ldots,s_q$,
namely, it depends on  the matrix $D$ and the kernel $U$ only.

\begin{lemma}
\label{lm:force_dens_part}
For all integers $q\ge 2$, $s_1,\ldots,s_q\in [q+2,2q+2]$ and a symmetric real matrix $D\in\RR^{q\times q}$,
there exists an $(s_1,\ldots,s_q)$-rooted quantum graph $T_{s_1,\ldots,s_q}$ satisfying the following.
Each constituent of $T_{s_1,\ldots,s_q}$ has $2q(q-1)$ non-root vertices, and
if $U$ is a $q$-step kernel such that
\begin{itemize}
\item the density of each part of $U$ is one of the diagonal entries of $D$, and
\item the density between each pair of the parts of $U$ is one of the off-diagonal entries of~$D$,
\end{itemize}
then there exists $c_0=c_0(D,U)\not=0$, which does not depend on $s_1,\ldots,s_q$, such that
$t_{\star}(T_{s_1,\ldots,s_q},U)$ is either $0$ or $c_0$ for all choices of roots and
it is non-zero if and only if 
\begin{itemize}
\item all roots from each of the $q$ groups of roots are chosen from the same part of $U$,
\item roots from different groups are chosen from different parts of $U$,
\item $D_{ii}$ is the density of the part of $U$ that the $i$-th group of roots is chosen from, and
\item $D_{ij}$ is the density between the parts of $U$ that the $i$-th and $j$-th groups of roots are chosen from.
\end{itemize}
\end{lemma}

\begin{proof}
Fix integers $q\ge 2$, $s_1,\ldots,s_q\in [q+2,2q+2]$,
and a matrix~$D$.
Let $Z_1$ be the set containing the values of diagonal entries of $D$ and
$Z_2$ the set containing the values of off-diagonal entries of $D$.
We next define a polynomial $p$, whose  $\binom{q+1}{2}$ are variables are indexed by pairs $ij$ with $1\le i\le j\le q$,
as follows:
\[p(x_{11},x_{12},\ldots,x_{q-1,q},x_{q,q})=
\OPleft\prod_{i=1}^q\prod_{z\in Z_1\setminus\{D_{ii}\}}(x_{ii}-z)\OPright
 \OPleft  \prod_{1\le i<j\le q}\prod_{z\in Z_2\setminus\{D_{ij}\}}(x_{ij}-z)\OPright.\]
Let $P_{s_1,\ldots,s_q}$ be the graph from Lemma~\ref{lm:steps}.
For $m_{ii}\in [0,|Z_1|]$, $i\in [q]$, and $m_{ij}\in [0,|Z_2|]$, $1\le i<j\le k$,
let $P^{m_{11},m_{12},\ldots,m_{q,q}}_{s_1,\ldots,s_q}$ be an $(s_1,\ldots,s_q)$-rooted quantum graph
obtained from $P_{s_1,\ldots,s_q}$ by adding 
arbitrary $m_{ij}$ edges joining roots in the $i$-th group and with the roots in the $j$-th group for $1\le i\le j\le q$ (without creating parallel edges).
The $(s_1,\ldots,s_q)$-rooted quantum graph $T_{s_1,\ldots,s_q}$
is obtained from the expansion of $p(x_{11},x_{12},\ldots,x_{q,q})$ into monomials
by replacing each monomial $x_{11}^{m_{11}}x_{12}^{m_{12}}\cdots x_{q,q}^{m_{q,q}}$
with $P^{m_{11},m_{12},\ldots,m_{q,q}}_{s_1,\ldots,s_q}$ (including the monomial $x_{11}^0\cdots x_{q,q}^0$).

Fix a $q$-step kernel $U$ such that
\begin{itemize}
\item the density of each part of $U$ belongs to $Z_1$, and
\item the density between each pair of the parts of $U$ belongs to $Z_2$.
\end{itemize}
Let $d_0=d_0(U)>0$ be the constant from Lemma~\ref{lm:steps}. Note that
$t_{\star}\left(T_{s_1,\ldots,s_q},U\right)=0$ unless
\begin{itemize}
\item all roots from each of the $q$ groups of roots are chosen from the same part of $U$,
\item roots from different groups are chosen from different parts of $U$,
\item $D_{ii}$ is the density of the part of $U$ that the $i$-th group of roots is chosen from, and
\item $D_{ij}$ is the density between the parts of $U$ that the $i$-th and $j$-th groups of roots are chosen from,
\end{itemize}
and if $t_{\star}\left(T_{s_1,\ldots,s_q},U\right)\not=0$,
then it is equal to
\[
c_0=d_0\OPleft\prod_{i=1}^q\prod_{z\in Z_1\setminus\{D_{ii}\}}(D_{ii}-z)\OPright
 \OPleft    \prod_{1\le i<j\le q}\prod_{z\in Z_2\setminus\{D_{ij}\}}(D_{ij}-z)\OPright\not=0.
\]
Hence, the $(s_1,\ldots,s_q)$-rooted quantum graph $T_{s_1,\ldots,s_q}$
has the properties given in the statement of the lemma.
\end{proof}

To prove the main result of this paper, we need the following well-known result,
which we state explicitly for reference.

\begin{lemma}
\label{lm:system}
For every $q\ge 1$ and reals $z_1,\ldots,z_q$,
the following system of equations has at most one solution $x_1,\ldots,x_q\in \mathbb R$ (up to a permutation of the values):
\[
\begin{array}{ccccccl}
x_1 & + & \cdots & + & x_q & = & z_1 \\
x_1^{2} & + & \cdots & + & x_q^{2} & = & z_2 \\
\vdots & & \vdots & & \vdots & & \vdots \\
x_1^{q} & + & \cdots & + & x_q^{q} & = & z_{q}.
\end{array}
\]
\end{lemma}

\begin{proof}
The system of equations gives the first $q$ power sums of $x_1,\dots,x_q$.
By Newton's identities (see e.g.~\cite[Equation (2.11$'$)]{Macdonald95}),
this determines the first $q$ elementary symmetric polynomials,
which are the coefficients of the polynomial $\prod_{i=1}^q(x+x_i)$.
Therefore any other solution $y_1,\dots,y_q$ of the system satisfies that $\prod_{i=1}^q(x+x_i)=\prod_{i=1}^q(x+y_i)$,
which yields the statement of the lemma because of the uniqueness of polynomial factorization.
\end{proof}

We are now ready to prove our main result,
which implies Theorem~\ref{thm:main-graphon} stated in Section~\ref{sec:intro}.

\begin{theorem}
\label{thm:main-kernel}
The following holds for every $q\ge 2$ and every $q$-step kernel $U$:
if the density of each graph with at most $4q^2-q$ vertices in a kernel $U'$ is the same as in $U$,
then the kernels $U$ and $U'$ are weakly isomorphic.
\end{theorem}

\begin{proof}
Fix a $q$-step kernel $U$.
Let $a_1,\ldots,a_q$ be the measures of the $q$ parts.
Further let $D\in\RR^{q\times q}$ be the matrix such that
$D_{ii}$ is the density of the $i$-th part of $U$ and
$D_{ij}$, $i\not=j$, is the density between the $i$-th and $j$-th part.

Consider a kernel $U'$ such that $t(H,U)=t(H,U')$ for all graphs with at most $4q^2-q$ vertices.
Since each constituent of the quantum graphs $Q_{q}$ and $Q_{q+1}$ from Lemma~\ref{lm:force_steps}
has $q(q+1)$ and $(q+1)(q+2)\le 4q^2-q$ vertices, respectively,
it holds that
$t(Q_{q},U')\not=0$ and $t(Q_{q+1},U')=0$ (as they are the same as the corresponding densities in $U$).
We conclude using Lemma~\ref{lm:force_steps} that $U'$ is a $q$-step kernel.

Let $R_{D_{11},\ldots,D_{qq}}$ be the quantum graph from the statement of Lemma~\ref{lm:force_dens_inner};
note that each constituent of $R_{D_{11},\ldots,D_{qq}}$ has $3q^2\le 4q^2-q$ vertices.
Since $t(R_{D_{11},\ldots,D_{qq}},U')=0$ (as the value is the same as for the kernel $U$),
Lemma~\ref{lm:force_dens_inner} yields that the density of each part of $U'$
is equal to one of the diagonal entries of $D$.
Similarly, Lemma~\ref{lm:force_dens_between} applied with the off-diagonal entries of $D$
yields that the density between any pair of parts of $U'$ is equal to one of the off-diagonal entries of $D$.
In addition,
the $(q+2,\ldots,q+2)$-rooted quantum graph $T_{q+2,\ldots,q+2}$ from Lemma~\ref{lm:force_dens_part} applied with the matrix $D$ satisfies
$t(\unlab{T_{q+2,\ldots,q+2}}{},U)\not=0$; thus
it holds that $t(\unlab{T_{q+2,\ldots,q+2}}{},U')\not=0$.
Hence, we derive using Lemma~\ref{lm:force_dens_part} that, after possibly permuting the parts of $U'$,
the density of the $i$-th part of $U'$ is $D_{ii}$ and
the density between the $i$-th and $j$-th parts of $U'$ is $D_{ij}$.

Let $d_0=d_0(U)>0$ be the constant from Lemma~\ref{lm:steps} for the kernel~$U$.
Observe that, for each  $k\in [0,q]$, the following holds for the rooted quantum graph $P_{q+k+2,q+2,\ldots,q+2}$ from Lemma~\ref{lm:steps}:
\[
t\left(\unlab{P_{q+k+2,q+2,\ldots,q+2}}{},U\right)=d_0(q-1)!\OPleft\prod_{j=1}^qa_j^{q+2}\OPright\OPleft\sum_{i=1}^q a_i^{k}\OPright.
\]
It follows that the following holds for every $k\in [q]$:
\[\sum_{i=1}^q a_i^{k}=\frac{q\cdot t\left(\unlab{P_{q+k+2,q+2,\ldots,q+2}}{},U\right)}{t\left(\unlab{P_{q+2,q+2,\ldots,q+2}}{},U\right)}.\]
Similarly, with $a_i'$ denoting the measure of the $i$-th part of $U'$, we obtain that
\[\sum_{i=1}^q \left(a'_i\right)^{k}=\frac{q\cdot t\left(\unlab{P_{q+k+2,q+2,\ldots,q+2}}{},U'\right)}{t\left(\unlab{P_{q+2,q+2,\ldots,q+2}}{},U'\right)}.\]
Hence, Lemma~\ref{lm:system} and
the assumption that the homomorphism densities of all graphs with at most $q(q+2)+q+2q(q-1)=3q^2+q\le 4q^2-q$ vertices are the same in $U$ and $U'$
implies that the multisets $a_1,\ldots,a_q$ and $a'_1,\ldots,a'_q$ are the same.

Let $c_0=c_0(D,U)\not=0$ be the constant from 
Lemma~\ref{lm:force_dens_part} for the kernel $U$ and
let $\Pi_D$ be the set of all permutations $\pi$ of the parts of $U$ such that
the densities inside the parts and between the parts in $U$ and after applying $\pi$ to the parts of $U$ are still as given by $D$.
Observe that it holds that
\[t(\unlab{T_{s_1,\ldots,s_q}}{},U)=c_0\sum_{\pi\in\Pi_D}\prod_{i=1}^q a_{\pi(i)}^{s_i}.\]

Let $p(x_1,\ldots,x_q)$ be the polynomial defined as
\[p(x_1,\ldots,x_q)=\OPleft\prod_{j=1}^q x_j^{q+2}\OPright\OPleft\prod_{i=1}^q \prod_{a\in \{a_1,\dots,a_q\}\setminus\{a_i\}} (x_i-a)\OPright.\]
Note that each variable in each monomial of $p$ has degree between $q+2$ and $2q+1$. Since each $a_i$ 
 is non-zero, we have for all $q$-tuples $(x_1,\dots,x_q)$ of reals with $\{x_1,\dots,x_q\}\subseteq \{a_1,\dots,a_q\}$ that
$p(x_1,\ldots,x_q)=0$ if and only if there exists $i\in [q]$ such that $x_i\not=a_i$. 
Let $T$ be the quantum graph obtained from the polynomial $p$
by expanding it and then replacing each monomial $x_1^{s_1}\cdots x_q^{s_q}$ with $\unlab{T_{s_1,\ldots,s_q}}{}$ (including
the monomial $x_1^0\cdots x_1^0$).
Note that
the number of vertices of each constituent of $T$ is at most $q(2q+1)+2q(q-1)=4q^2-q$ and
\[t(T,U)=c_0\sum_{\pi\in\Pi_D}p(a_{\pi(1)},\ldots,a_{\pi(q)}).\]
In particular, it holds that $t(T,U)\not=0$ and so $t(T,U')\not=0$.
Along the same lines, we obtain that
\[t(T,U')=c_0'\sum_{\pi\in\Pi'_D}p(a'_{\pi(1)},\ldots,a'_{\pi(q)}),\]
where $c_0'=c_0(D,U')\not=0$ is the constant from Lemma~\ref{lm:force_dens_part} for the kernel $U'$ and
$\Pi'_D$ is the set of all permutations $\pi$ of the parts of $U'$ such that
the densities of the parts and between the parts after applying $\pi$ are as given by $D$.
Since it holds that $t(T,U')\not=0$,
the set $\Pi'_D$ is non-empty.
It follows that $\Pi'_D$ contains a permutation $\pi$ such that $a'_{\pi(i)}=a_i$ for all $i\in [q]$,
which implies that the kernels $U$ and $U'$ are weakly isomorphic.
\end{proof}

\section{Parts with different degrees}
\label{sec:different}

In this section,
we show that a $q$-step kernel such that
its vertices contained in different parts have different degrees
is forced by graphs with at most $2q+1$ vertices.

\begin{theorem}
\label{thm:different}
The following holds for every $q\ge 2$ and every $q$-step kernel $U$ such that
the degrees of vertices in different parts are different:
if the density of each graph with at most $2q+1$ vertices in a kernel $U'$ is the same as in $U$,
then the kernels $U$ and $U'$ are weakly isomorphic.
\end{theorem}

\begin{proof}
Fix $q\ge 2$, a $q$-step kernel $U$ and
a kernel $U'$ such that $t(H,U)=t(H,U')$ for every graph $H$ with at most $2q+1$ vertices.
For $i\in [q]$, let $A_i$ be the $i$-th part of $U$,
$a_i$ be the measure of $A_i$, and
let $d_i$ be the common degree of the vertices contained in~$A_i$.

Let $K_1^{\bullet}$ and $K_2^{\bullet}$ be the $1$-rooted graphs obtained from $K_1$ and $K_2$, respectively,
by choosing one of their vertices to be the root.
Note that $t_x(K_2^{\bullet}-d K_1^{\bullet},V)=0$ if and only if the degree of $x$ in a kernel $V$ is $d$.
It follows that a kernel $V$ satisfies that
\begin{equation}
t\left(\unlab{\prod_{i\in [q]}\left(K_2^{\bullet}-d_i K_1^{\bullet}\right)^2}{},V\right)=0
\label{eq:diff1}
\end{equation}
if and only if the degree of almost every vertex of $V$ is one of the numbers $d_1,\ldots,d_q$.
Next observe that if a kernel $V$ satisfies \eqref{eq:diff1} and
\begin{equation}
t\left(\unlab{\prod_{i\in [q]\setminus\{k\}}\left(K_2^{\bullet}-d_i K_1^{\bullet}\right)}{},V\right)=
a_k \prod_{i\in [q]\setminus\{k\}}\left(d_k-d_i\right)
\label{eq:diff2}
\end{equation}
for $k\in [q]$, then the measure of the set of vertices of $V$ with degree equal to $d_k$ is $a_k$.
Since $U$ satisfies \eqref{eq:diff1} and \eqref{eq:diff2} for every $k\in [q]$ and
the graphs appearing in \eqref{eq:diff1} and \eqref{eq:diff2} have at most $2q+1$ and $q$ vertices, respectively,
the vertex set of the kernel $U'$ can be partitioned into $q$ (measurable) sets $A'_1,\ldots,A'_q$ and a null set $A'_0$ such that
the measure of $A'_k$ is $a_k$ and
all vertices contained in $A'_k$ have degree equal to $d_k$ for every $k\in [q]$.

Let $G^{\bullet\bullet}$, $G^{\circ\bullet}$ and $G^{\bullet\circ}$ be the following $2$-rooted graphs:
$G^{\bullet\bullet}$ consists of two isolated roots only,
$G^{\circ\bullet}$ is obtained from $G^{\bullet\bullet}$ by adding a non-root vertex adjacent to the first root, and
$G^{\bullet\circ}$ is obtained from $G^{\bullet\bullet}$ by adding a non-root vertex adjacent to the second root.
For $k,\ell\in [q]$, let $H_{k\ell}$ be the $2$-rooted quantum graph defined as
\[H_{k\ell}=\OPleft\prod_{i\in [q]\setminus\{k\}}\left(G^{\circ\bullet}-d_i G^{\bullet\bullet}\right)\OPright\times
            \OPleft\prod_{j\in [q]\setminus\{\ell\}}\left(G^{\bullet\circ}-d_j G^{\bullet\bullet}\right)\OPright,\]
and observe that
\[
t_{xy}\left(H_{k\ell},U'\right)=\OPleft\prod_{i\in [q]\setminus\{k\}}\left(d_k-d_i\right)\OPright\OPleft
\prod_{j\in [q]\setminus\{\ell\}}\left(d_\ell-d_j\right)\OPright
\]
if the degree of $x$ is $d_k$ and the degree of $y$ is $d_\ell$, and
$t_{xy}\left(H_{k\ell},U'\right)=0$ otherwise.
In particular, it follows that
\begin{equation}
t\left(\unlab{H_{k\ell}}{},U'\right)=a_ka_\ell\OPleft\prod_{i\in [q]\setminus\{k\}}\left(d_k-d_i\right)\OPright\OPleft                                    \prod_{j\in [q]\setminus\{\ell\}}\left(d_\ell-d_j\right)\OPright.
\label{eq:diff3}
\end{equation}
Note that each constituent of the $2$-rooted quantum graph $H_{k\ell}$ has at most $2q$ vertices.
Let $H'_{k\ell}$ be the $2$-rooted quantum graph obtained from $H_{k\ell}$ by joining the two roots in each of its constituents by an edge. Similarly as above, one can show that
\begin{equation}
t\left(\unlab{H'_{k\ell}}{},U'\right)=\OPleft\int_{A'_k\times A'_\ell} U'(x,y)\dd x\dd y\OPright
\OPleft\prod_{i\in [q]\setminus\{k\}}\left(d_k-d_i\right)\OPright\OPleft                                    \prod_{j\in [q]\setminus\{\ell\}}\left(d_\ell-d_j\right)\OPright.
\label{eq:diff3'}
\end{equation}
Using \eqref{eq:diff3} and \eqref{eq:diff3'}, we obtain that
$$
\frac{\int_{A'_k\times A'_\ell} U'(x,y)\dd x \dd y}{a_ka_{\ell}}=
\frac{t\left(\unlab{H'_{k\ell}}{},U'\right)}{t\left(\unlab{H_{k\ell}}{},U'\right)}=
\frac{t\left(\unlab{H'_{k\ell}}{},U\right)}{t\left(\unlab{H_{k\ell}}{},U\right)}=
\frac{\int_{A_k\times A_\ell} U(x,y)\dd x \dd y}{a_ka_{\ell}},
$$
i.e., the average density between the parts $A'_k$ and $A'_\ell$ in the kernel $U'$ is the same as the density between the parts $A_k$ and $A_\ell$ in the kernel $U$.

We now recall
each step kernel is the unique (up to weak isomorphism) minimizer of the density of $C_4$
among all kernels with the same number of parts of the same measures and the same density between them.
This statement for step graphons with parts of equal measure appears in~\cite[Lemma 11]{CooKM18} and
the same proof applies for kernels with parts not necessarily having the same sizes;
also see~\cite[Propositions 14.13 and 14.14]{Lov12} for related results.
Since $t(C_4,U)=t(C_4,U')$,
it follows that $U'$ is weakly isomorphic to $U$.
\end{proof}

\section{Concluding remarks}

Theorem~\ref{thm:main-kernel} asserts that every $q$-step kernel
is forced by graphs with at most $4q^2-q$ vertices. 
We do not know whether it suffices to consider homomorphism densities of graphs with $o(q^2)$ vertices,
both in the case of kernels and in the more restrictive case of graphons.
We leave this as an open problem.

We finish by establishing that it is necessary to consider graphs with the number of vertices linear in~$q$.
The argument is similar to that used in analogous scenarios, e.g., in~\cite{ErdLS79,GleGKK15}.
For reals $a_1,\ldots,a_q\in (0,1)$ such that $a_1+\cdots+a_q<1$,
let $U_{a_1,\ldots,a_q}$ be the $(q+1)$-step graphon with parts of measures $a_1,\ldots,a_q$ and $1-a_1-\cdots-a_q$ such that
the graphon $U_{a_1,\ldots,a_q}$ is equal to one within each of the first $q$ parts and to zero elsewhere.
Observe that if $H$ is a graph, which  consists of $k$ components with $n_1,\ldots,n_k$ vertices after the removal of isolated vertices, then
\[t\left(H,U_{a_1,\ldots,a_q}\right)=\prod_{i=1}^k\sum_{j=1}^q a_j^{n_i}=\prod_{i=1}^kt\left(K_{n_i},U_{a_1,\ldots,a_q}\right).\]
It follows that if
\begin{equation}
t\left(K_{\ell+1},U_{a_1,\ldots,a_q}\right)=t\left(K_{\ell+1},U_{a'_1,\ldots,a'_q}\right)\quad\mbox{ for every $\ell\in [q-1]$,}
\label{eq:Kl}
\end{equation}
then the homomorphism density of every graph with at most $q$ vertices
is the same in $U_{a_1,\ldots,a_q}$ and in $U_{a'_1,\ldots,a'_q}$. 
When $f(a_1,\dots,a_q)=(t(K_{\ell+1},U_{a_1,\ldots,a_q}))_{\ell=1}^{q-1}$ is viewed as a function of $a_1,\ldots,a_q$,
then its Jacobian matrix $J$ with respect to the first $q-1$ coordinates is 
\begin{equation}\label{eq:J}
\left[\begin{array}{ccc}
 2a_1&\cdots& 2 a_{q-1}\\
 3a_1^2&\cdots & 3 a_{q-1}^2\\
 \vdots & &\vdots\\
 q a_1^{q-1}&\cdots & q a_{q-1}^{a-1}
 \end{array}\right]
=
\left[\begin{array}{ccc}
  2\\
  &\ddots\\
  && q
\end{array}\right]
\left[\begin{array}{ccc}
  1&\cdots& 1\\
  a_1&\cdots & a_{q-1}\\
  \vdots & &\vdots\\
  a_1^{q-2}&\cdots & a_{q-1}^{a-2}
\end{array}\right]
\left[\begin{array}{ccc}
  a_1\\
  &\ddots\\
  && a_{q-1}
\end{array}\right].
\end{equation}
Fix any distinct positive reals $a_1,\dots,a_q$ with sum less than~$1$.
Note that the middle matrix in~\eqref{eq:J} is the Vandermonde matrix of $(a_1,\dots,a_{q-1})$ and thus the Jacobian matrix $J$ is non-singular.
By the Implicit Function Theorem, for every $a_q'$ sufficiently close to $a_q$ there is a vector $(a'_1,\ldots,a'_{q-1})$ close to $(a_1,\dots,q_q)$ such that \eqref{eq:Kl} holds.
By making $a_q'$ sufficiently close but not equal to $a_q$, we can ensure that $a_q'\not \in\{a_1,\dots,a_q\}$ and that all elements $a_i'$ are positive and sum to less than~$1$.
Thus we obtain two $(q+1)$-step graphons, namely $U_{a_1,\dots,a_q}$ and $U_{a'_1,\dots,a'_q}$, that
have the same homomorphism density of every graph with at most $q$ vertices but are not weakly isomorphic;
the latter can be established by e.g.\ applying the proof of Theorem~\ref{thm:main-kernel} to these two graphons (alternatively,
it also follows from the general analytic characterization of weakly isomorphic kernels~\cite[Theorem~13.10]{Lov12}).

\section*{Acknowledgement}

The authors would like to thank both anonymous reviewers for their comments that
have improved the presentation of the results in the paper.

\bibliographystyle{bibstyle}
\bibliography{stepforce}
\end{document}